\documentclass[11pt, a4paper]{amsart}
\usepackage{amsmath}
\usepackage{tikz}
\usepackage{amssymb}
\usepackage{amsthm}
\usepackage{tikz-cd}
\usepackage[all]{xy}
\usepackage{amsfonts}
\usepackage{mathrsfs}
\usepackage{hyperref}
\usepackage{xcolor}
\usepackage{MnSymbol}
\usepackage{mathtools}
\usepackage{subfig, caption}
\usepackage{wrapfig}
\usepackage{pgfplots}
\pgfplotsset{width=10cm,compat=1.9}
\usepackage[utf8]{inputenc}
\usepackage{graphicx} 
\usepackage{cleveref}
\usepackage{kotex}

\input{header_labeling.tex}

\newcommand{\DD}{\mathbb{D}}

\newcommand{\HH}{\mathbb{H}}

\newcommand{\NN}{\mathbb{N}}

\newcommand{\RR}{\mathbb{R}}

\newcommand{\ZZ}{\mathbb{Z}}

\newcommand{\cB}{\mathcal{B}}
\newcommand{\cC}{\mathcal{C}}

\newcommand{\cF}{\mathcal{F}}

\newcommand{\sU}{\mathscr{U}}

\newcommand{\Homeo}{\operatorname{Homeo}}
\newcommand{\Homeop}{\operatorname{Homeo}_+}

\newcommand{\Isomp}{\operatorname{Isom_+}}

\newcommand{\Fix}{\operatorname{Fix}}

\newcommand{\PSL}{\operatorname{PSL}_2}

\newcommand{\Mod}{\operatorname{Mod}}

\newcommand{\closure}[1]{\overline{#1}}

\newcommand{\vol}{\mathrm{vol}}
\newcommand{\supp}{\operatorname{supp}}

\title{Matsumoto dichotomy on foliated $S^1$-bundles}

\author[Kim]{KyeongRo Kim}
\address{\hskip-\parindent
Research institute of Mathematics\\
Seoul National University\\
GwanAkRo 1, Gwanak- Gu, Seoul 08826, Korea}
\email{kyeongrokim14@gmail.com}

\author[Lee]{Hongjun Lee}
\address{\hskip-\parindent
Department of Mathematical Sciences\\
KAIST\\
291 Daehak-Ro Yuseong-Gu, Daejeon, 34141, South Korea}
\email{hj\_drhouse@kaist.ac.kr}

\date{\today}

\begin{document}

\begin{abstract} 
Given an ergodic harmonic measure on a foliated circle bundle over a closed hyperbolic manifold, Matsumoto constructed a map from the fiber circle to the space of nonempty closed subsets of the boundary sphere of the universal cover of the base manifold. This map is well-defined at almost every point.
Also, the map is equivariant under two actions of the fundamental group of the base manifold: the holonomy action on the fiber and the action on the space of closed subsets induced by the boundary sphere action.

Matsumoto established a dichotomy for these maps, which corresponds to a dichotomy of ergodic harmonic measures. (Indeed, the Matsumoto dichotomy also concerns ergodic harmonic measures on compact hyperbolic laminations.) The dichotomy says that a Matsumoto map either maps each point to a singleton (Type I) or to the entire sphere (Type II).

In this paper, we study actions of closed hyperbolic manifold groups on the circle in the context of the Matsumoto dichotomy. We essentially show that the suspension of any action with a non-discrete image cannot admit a Matsumoto map of type I under the condition of a uniformly bounded number of fixed points. As a consequence, we address a question posed in Matsumoto's paper.
\end{abstract}

\maketitle

\section{Introduction}\label{Sec:intro}
Let $S$ be a closed hyperbolic surface. Say that $\Gamma$ is a discrete subgroup of $\Isomp(\HH^2)$ such that $S=\HH^2/\Gamma$ where $\HH^2$ is the hyperbolic plane and $\Isomp(\HH^2)$ is the group of orientation preserving isometries of $\HH^2$. Then, $\Gamma$ acts continuously on the circle at infinity $S_\infty^1$, that is, there is an embedding $\iota:\Gamma \to \Homeop(S^1)$.
When $\varphi$ is a pseudo-Anosov mapping class in $\Mod(S)$, any lifting of each representative of $\varphi$ can be continuously extended to the circle at infinity. Say that an element $\phi$ in $\Homeop(S^1)$ is obtained by restricting such an extension.
The subgroup of $\Homeop(S^1)$, generated by $\iota(\Gamma)$ and $\phi$, is isomorphic to the fundamental group of the mapping torus $M_\varphi$ of $\varphi$. 
Hence, we can see that  $\pi_1(M_\varphi)$ acts faithfully on the circle. 
We say that  $\rho_{univ}:\pi_1(M_\varphi)\to \Homeop(S^1)$ is the action.

On the other hands, by Thurston's hyperbolization theorem for a mapping torus, \cite[Theorem~0.1]{Thurston86}, $M_\varphi$ admits a complete hyperbolic metric. Hence, the universal cover of $M_\varphi$ is the hyperbolic space $\HH^3$, and the fundamental group $\pi_1(M_\varphi)$ acts on the sphere at infinity $S_\infty^2$.
Say that $\rho_\infty: \pi_1(M_\varphi)\to \Homeop(S^2)$ is the boundary action. 
Cannon and Thurston \cite{CannonThurston07} found an interesting connection between the actions $\rho_{univ}$ and $\rho_\infty$.
They showed that there is a contiuous surjection $m_{CT}:S^1\to S^2$ (so-called a Cannon-Thurston map) that is $(\rho_{univ},\rho_{\infty})$-equivariant, namely, $m_{CT} \circ \rho_{univ}(g)=\rho_{\infty}(g)\circ m_{CT}$ for all $g\in \pi_1(M_\varphi)$.
See \cite{CannonThurston07} for the more detailed exposition about $m_{CT}$.   

More generally, Thurston \cite{ThurstonSlitheringI} and Calegari-Dunfield \cite{CalegariDunfield}  showed that the fundamental group of an atoroidal 3-manifold $M$ with a taut foliation $\cF$ acts faithfully on a `universal' circle.
Such an action is called a \emph{universal circle action} $\rho_{univ}$, which encodes the topology of $\cF$. 
Here, $\rho_{univ}$ coincides with the previous $\rho_{univ}$ when $\cF$ is a fibration.
Thurston's construction of a universal circle action has been generalized under the various topological structures, e.g. pseudo-Anosov flows, quasi-geodesic flows, essential laminations, veering triangulations and so on.
For instance, see \cite{Calegari07}, \cite{Fenley16}, \cite{Frankel13} and \cite{FrankelSchleimerSegerman}.  

In fact, a universal circle action is not uniquely determined by $\cF$. Also, a universal circle action is not the only way that $\pi_1(M)$ acts on the circle.
For some class of hyperbolic three manifolds $M$, we can construct a faithful action $\pi_1(M)\to \PSL(\RR)$.
Recall that $\PSL(\RR)$ is same with $\Isomp(\HH^2)$ and the action of $\PSL(\RR)$ on the circle at infinity induces a natural embedding $\PSL(\RR)\to \Homeop(S^1)$. 
See \refexa{GaloisConj} for the detailed construction.

Motivated by the result of Cannon and Thurston, we discuss whether there is a connection between  $\rho_\infty$ and any given action $\pi_1(M)\to \Homeop(S^1)$, which is not necessarily faithful.
More generally, let $\Gamma$ be a discrete subgroup of $\Isomp(\HH^n)$ such that $M=\HH^n/G$ is a closed manifold, and $\rho:\Gamma \to \Homeop(S^1)$ a homomorphism. Say that $\rho_\infty:\Gamma\to \Homeop(S^{n-1})$ is the  action of $\Gamma$ on the sphere at infinity $S_\infty^{n-1}$.
In this paper, we discuss whether $\rho$  is  related to $\rho_\infty$ in any interesting way.

To this end, we study  \emph{Matsumoto maps}, introduced in \cite{Matsumoto12}. A Matsumoto map $\frak{m}$ (under our setting) is a $(\rho,\rho_\infty)$-equivariant measurable map from $S^1$ to the space of nonempty closed subsets
of $S_\infty^{n-1}$. 
It is given by some ergodic harmonic measure on the \emph{suspension} of $\rho$. 
Matsumoto showed that either $\frak{m}(x)$ is a singleton for almost every $x$ in $S^1$ (Type I) or $\frak{m}(x)=S_\infty^{n-1}$ for almost every $x$ in $S^1$ (Type II), which we call the \emph{Matsumoto dichotomy}. 
See \refsec{MatsumotoDichotomy} for the precise statement of the Matsumoto dichotomy.

In the Type I case, we can obtain an `almost' equivariant map from $S^1$ to $S_\infty^{n-1}$, considering a singleton as a point, similar to a Cannon-Thurston map.
Recently, Adachi, Matsuda and Nozawa \cite{AdachiMatsudaNozawa} constructed a Matsumoto map (of Type I)  for a circle action of a punctured hyperbolic surface group with the maximal (relative) Euler number.  
This provided an alternative proof of the rigidity of such a circle action of and Burger–Iozzi–Wienhard \cite{BurgerIozziWienhard}. 

In this paper, we show that no non-discrete action can admit a  Matsumoto map of Type I under the condition of at most $N$ fixed points.
Obviously, the above example of a hyperbolic mapping torus (\refexa{fibering}) and any $\PSL(\RR)$-representation satisfy this condition.

\begin{thm}\label{Thm:indiscreteType2}
Let $\Gamma$ be a subgroup of  $\Isomp(\HH^n)$ such that $\HH^n/\Gamma$ is a closed manifold, and $\rho:\Gamma\to \Homeop(S^1)$ an action of $\Gamma$ on $S^1$.
Assume that $\rho$ satisfies one of the following conditions:
\begin{itemize}
    \item $\rho$ is non-faithful or
    \item $\rho$ is a faithful action satisfying the following properties:
\begin{itemize}
    \item $\rho(\Gamma)$ is non-discrete in $\Homeop(S^1)$ and
    \item there is a number $N\in \NN$ such that $\rho(g)$ has at most $N$ fixed points for all $g\in \Gamma\setminus \{id\}$.
\end{itemize}
\end{itemize}
Then, any ergodic harmonic measure of the suspension foliation of $\rho$ is of Type~II.
\end{thm}

This result addresses the following question posed in \cite[Question 6.9]{Matsumoto12}.
\begin{ques}[Matsumoto]
    For an injective homomorphism $\rho$ from the fundamental group $\Gamma$ of a closed hyperbolic surface to $\PSL(\RR)$ with dense image, is any ergodic harmonic measure of the suspension of $\rho$ type II?
\end{ques}

By combining \refthm{indiscreteType2} with \cite[Example~6.5]{Matsumoto12} and \cite[Proposition~6.6]{Matsumoto12}, we can complete the classification of ergodic harmonic measures on the foliated $S^1$-bundle over a closed hyperbolic surface $S$, given as a suspension of a $\PSL(\RR)$-representation of $\pi_1(S)$.
\begin{cor}\label{Cor:classificationOverSurface}
    Let $\Gamma$ be a subgroup of $\PSL(\RR)$ such that $\HH^2/\Gamma$ is a closed hyperbolic surface and let $\cF$ the suspension foliation of a homomorphism $\rho:\Gamma\to \PSL(\RR)$.
     Then, the following holds.
    \begin{itemize}
         \item If $\rho$ is discrete and faithful, then there is a unique ergodic harmonic measure of $\cF$ that is of Type~I ;
        \item If $\rho$ is either with non-discrete image (equivalently, dense image) or non-faithful, then any ergodic harmonic measure of $\cF$ is of  Type~II.        
    \end{itemize}
\end{cor}

\section{Matsumoto Dichotomy}\label{Sec:MatsumotoDichotomy}
For our purpose, we briefly recall the Matsumoto dichotomy only for ergodic harmonic measures on compact $C^2$ foliation even though the original Matsumoto dichotomy addresses ergodic harmonic measures on compact $C^2$ laminations.
We refer to \cite{Matsumoto12} for the original version of the Matsumoto dichotomy.
Also, see \cite{CandelConlon00} for the basic terminology about foliations. 

\subsection{Compact hyperbolic $C^2$ foliations}
In this paper, a codimension $k$ foliation of $\cF$ of a closed $n$-manifold $M$ is said to be \emph{of class $C^2$} if there is an foliated atlas $\{E_\alpha, \varphi_\alpha\}_{\alpha\in \sU}$ associated with $(M,\cF)$ such that $\varphi_\alpha$ is a homeomorphism from $E_\alpha$ to $U_\alpha \times Z_\alpha$ for some open disks $U_\alpha\subset \RR^{n-k}$ and $Z_\alpha \subset \RR^k$ and whenever $E_\alpha \cap E_\beta \neq \emptyset$,  the transition map $\psi_{\beta \alpha}=\varphi_\beta \circ \varphi_\alpha^{-1}$ is of the form 
\[
\psi_{\beta \alpha}(u,z)=(\psi_{\beta \alpha}^1(u,z), \psi_{\beta \alpha}^2(z))
\]
where $\psi_{\beta \alpha}^i$ are continuous, and $\psi_{\beta \alpha}^1$ is of class $C^2$ with respect to the first coordinate $u$ and its first and second derivatives  with respect to $u$ are also  continuous in $z$. 

A \emph{compact $C^2$ foliation} is a triple $(M,\cF,g)$ if $\cF$ is of class $C^2$ and $g$ is a \emph{leafwise Riemannian metric of class $C^2$}, which is a continous field of leafwise metric tensor in $(M,\cF)$ such that its first and second leafwise derivatives are continuous. 
In particular, it is said to be \emph{hyperbolic} if $g$ is \emph{leafwise hyperbolic}, that is, $g$ has constant sectional curvature $-1$ on each leaf.

A typical example of a compact hyperbolic $C^2$ foliation is given by the \emph{suspension} of an action of the fundamental group of a closed hyperbolic manifold  on a closed manifold $Z$. 
More precisely, say that $\Gamma$ is  a discrete subgroup of $\Isomp(\HH^n)$ such that  $B=\HH^n/\Gamma$ is  a closed manifold, and  assume that there is an action $\rho: \Gamma \to \Homeo(Z)$.
The \emph{suspension} of $\rho$ is defined as the quotient space $B\times_\rho Z$ of $\HH^n\times Z$ by the diagonal action of $\Gamma$, that is, $B\times_\rho Z=(\HH^n\times Z)/\Gamma$ where $g\cdot (x,z)=(g(x), \rho(g)(z))$ for $g\in \Gamma$.     
Then, there is a natural foliation $\cF_\rho$ on $B\times_\rho Z$, called the \emph{suspension foliation} of $\rho$,  induced from the horizontal foliation $\{\HH^n\times \{z\}: z\in Z\}$ on $\HH^n\times Z$. 
From the construction, we can see that the suspension foliation $\cF_\rho$ admits a leafwise hyperbolic metric. 
Whenever we mention the \emph{standard leafwise hyperbolic metric} on $\cF_\rho$, it refers to the leafwise hyperbolic metric obtained above.  

In this paper, we consider the case where $Z=S^1$.
In this case, the suspension $B\times_\rho S^1$ is a foliated circle bundle over $B$.
See \cite{ThurstonSlitheringII} for a nice explanation of the relationship between circle actions of a closed hyperbolic manifold group and foliated circle bundles over the manifold.

\subsection{Harmonic measures}
Let $(M,\cF,g)$ be a compact $C^2$ foliation. 
For each continuous function $f$ on $M$ that is leafwise $C^2$, the leafwise Laplacian $\Delta f$ is well-defined. 
A \emph{harmonic measure} $m$ on $M$ is a probability measure on $M$ such that for any continuous leafwise $C^2$ function $f$ on $M$, \[m(\Delta f): = \int_M\Delta f dm =0.\] 
A harmonic measure $m$ is \textit{ergodic} if whenever it is written as a nontrivial linear combination of two harmonic measures $m_1$ and $m_2$ then $m = m_1 = m_2$.
Garnett \cite{Garnett} showed that a compact $C^2$ lamination always admits a harmonic measure. A simple proof of this fact was also provided by Candel \cite{Candel03}.

Here is an important local structure theorem of harmonic measure on compact $C^2$ foliation:
\begin{thm}[Local Structure of Harmonic Measure]\label{Thm:localStructure}
Let $(M,\cF,g)$ be  a compact $C^2$ foliation. 
Assume that $m$ is a harmonic measure on $(M,\cF,g)$.
For any foliated local chart $U\times Z$, associated with $(M,\cF)$, there is a pushforward measure $\nu: = (\pi_2)_{\ast} m$ on $Z$ for the projection $\pi_2:U\times Z\to Z$
onto the second coordinate and a
leafwise harmonic function $h:U\times Z\to \RR$ 
with the following properties. 
\begin{enumerate}
    \item $h$ is positive and $m$-measurable.

    \item For $\nu$-a.e.\ $z \in Z$, the restriction of 
$h$ to the plaque $U\times z$
is harmonic and $h\mathrm{vol}$ is a probablity measure of the plaque where $\mathrm{vol}$ is the volume form on each leaf.

    \item For any continuous function with 
support in $U\times Z$, we have
$$
m(f) = \int_{U\times Z} f\ dm=\int_Z\int_{U\times{z}}f(u,z)h(u,z)d{\rm vol}(u)d\nu(z).
$$
\end{enumerate}
In other words, we may write
\[dm = h(u,z)d\vol (u)d\nu(z).\]
\end{thm}

Let $(M,\cF, g)$ be a compact  $C^2$ foliation.
When two foliated local charts $U\times Z$ and $U'\times Z'$ intersect, by applying the local structure theorem to each local chart, we have 
\[dm|_{U\times Z} = h\mathrm{vol}d\nu \text{ and } dm|_{U'\times Z'} = h'\mathrm{vol}d\nu'\]
and then
\[
h'/h=d\nu/d(\beta \nu')
\]
where $\beta$ is the holonomy map from a part of $Z'$ to $Z$. 
This shows that $\nu$ and $\nu'$ are equivalent via the holonomy map. Also, on each plaque, $h'$ is just a constant multiple of $h$.
Therefore, after taking a suitable normalization, one can prolong $h$  along a chain of plaques. 
Continuation of this process gives rise to the maximal prolongation, namely, 
the function $h$ is well-defined as a positive harmonic function on the holonomy cover $\hat{L}$ of the leaf $L$ containing the chains of plaques.
In particular, two such functions on $\hat{L}$ starting from the different plague can differ up to scalar multiple. We summarize as the following theorem.
In what follows, ``for an $m$-a.e. leaf $L$" means ``for any leaf $L$ in $M^*$",  for some saturated conull set $M^*$.
\begin{thm}[\cite{Matsumoto12}]
\label{Thm:wellDefine}
Let $m$ be a harmonic measure on a compact $C^2$ foliation. Then the following hold:
    \begin{enumerate}
        \item For an $m$-a.e. leaf $L$, the function $h$ has a well-defined prolongation as a positive harmonic function on the holonomy cover $\hat{L}$. On $\hat{L}$, two such functions which start from different plagues are unique up to a positive constant multiple.
        \item Given a path in $L$, the ratio of $h$ at the initial point and the terminal point of any lift of the path to $\hat{L}$ is constant.
    \end{enumerate}
\end{thm}
The immediate corollary is the following.
\begin{cor}[\cite{Matsumoto12}]
\label{Cor:equivariance}
    Let $\Gamma$ be the deck group of the universal covering $\tilde{L}\to L$. Then for any $\gamma\in\Gamma$, $h\circ\gamma$ is a constant multiple of $h$.
\end{cor}

The function $h$ associated with a harmonic measure $m$, defined in \refthm{wellDefine}, is called the \emph{characteristic function} of $m$. 
\begin{rmk}
By \refthm{wellDefine}, we stated above, the characteristic harmonic function is defined only up to a positive constant multiple.
\end{rmk}

\begin{defn}
    A harmonic measure $m$ on a compact $C^2$ foliation is called \emph{completely invariant} if the characteristic harmonic functions are constant on (the holonomy covers of) $m$-a.e. leaves. 
\end{defn}
\begin{rmk}
   If $m$ is a completely invariant harmonic measure on the suspension $B\times_\rho S^1$ for some closed manifold $B$ and some action $\rho:\pi_1(B)\to \Homeop(S^1)$, then $m$ corresponds to an invariant measure on $S^1$ under the action $\rho$.
\end{rmk}

\subsection{Matsumoto Dichotomy}

Let $m$ be a harmonic measure on a compact hyperbolic $C^2$ foliation $(M,\cF,g)$.
For the convenience, from now on, we use the Poincar\'e ball model $\DD^n$ instead of the half space model $\HH^n$.
Note that the universal cover of each leaf is identified with the hyperbolic space $\Bbb D^{d+1}$ for some positive integer $d$.

For an $m$-a.e. leaf $L$, we can define a measure class $[\mu_L]$ on the boundary $S^d_\infty$ of the universal cover $\tilde{L}=\DD^{d+1}$ as follows:
by \refthm{wellDefine}, the characteristic harmonic function $h$ of $m$ is defined on $\tilde{L}=\Bbb D^{d+1}$. Choose a base point $\tilde{x}\in\Bbb D^{d+1}$. Since $h$ is well-defined up to constant multiple by \refcor{equivariance}, we may assume that $h(\tilde{x}) = 1$. For any point $\xi$ of the ideal boundary $S^d_\infty$, we let $k_{\xi} = \exp(-d\cdot B_{\xi})$ where $B_{\xi}:\Bbb D^{d+1}\to\Bbb R$ is the Busemann function corresponding to $\xi$ such that $B_{\xi}(\tilde{x}) = 0$. Note that $k_{\xi}$ is the (minimal) positive harmonic function on $\Bbb D^{d+1}$ corresponding to $\xi$ normalized to take value $1$ at $\tilde{x}$. It is known that $k_\xi$ is the Poisson kernel of $\Bbb D^{d+1}$ and, by the Dirichlet correspondence, there is a unique probability measure $\mu_{\tilde{x}}$ on $S^d_{\infty}$ such that
$$h = \int_{S^d_\infty} k_{\xi}d\mu_{\tilde{x}}(\xi).$$
So far, we have associated each point $\tilde{x}\in \DD^d$ with the unique measure $\mu_{\tilde{x}}$ on $S^d_\infty$. Although the measure $\mu_{\tilde{x}}$ depends on the point $\tilde{x}$, its equivalence class of measure $[\mu_L]$ is an invariant of the leaf $L$.
Therefore, for an $m$-a.e. leaf $L$ , we can associate the unique class of measure $[\mu_L]$ on $\partial \tilde{L} = S^d_{\infty}$. 

\begin{defn}
    Let $(M,\cF,g)$ be a compact hyperbolic $C^2$ foliation.
    A harmonic measure $m$ on  $(M,\cF,g)$ is of \emph{Type I} if for an $m$-a.e. leaf $L$, the support of the associated measure class $[\mu_L]$ on $S^d_\infty$ is a singleton, and of \emph{Type II} if the support of $[\mu_L]$ is the whole $S^d_\infty$.
\end{defn}
Matsumoto dichotomy in the case of compact hyperbolic $C^2$ foliation is the following dichotomy of an ergodic harmonic measure.
\begin{thm}[\cite{Matsumoto12}]\label{Thm:MatsumotoDichotomy}
Any ergodic harmonic measure on a compact hyperbolic $C^2$ foliation is either of Type I or Type II.
\end{thm}

The above theorem says that for an ergodic harmonic measure $m$, either the support $K_L = \operatorname{supp}([\mu_L])$, called the \emph{characteristic set} of $L$, is a singleton for any $m$-a.e. leaf $L$, or the total space $S^d_\infty$ for any $m$-a.e. leaf $L$.

\section{The Matsumoto Map}
In this section, we recall the notion of the \textit{Matsumoto map} and investigate its continuity in the case of a $S^1$-bundle over a closed hyperbolic manifold, in order to prove \refthm{indiscreteType2}.

Let $B_\Gamma = \Bbb D^n/\Gamma$ be a closed hyperbolic $n$-manifold where $\Gamma$ is a discrete subgroup of $\mathrm{Isom}_+(\Bbb D^n)$ and let $\rho:\Gamma \to \Homeop(Z)$ an action of $\Gamma$ on some closed manifold $Z$.
Say that $M_\rho$ is the suspension $B_\Gamma\times_\rho Z$. Then, $M_\rho$ is equipped with the suspension foliation $\cF_\rho$. Suppose that $m$ is an ergodic harmonic measure on the compact hyperbolic $C^2$ foliation $(M_\rho,\cF_\rho, g)$ where $g$ is the standard leafwise hyperbolic metric.
Recall that $B_\Gamma\times_\rho Z=(\DD^n \times Z) /\Gamma$. 
We denote by $L_z$ the leaf of $\cF_\rho$, covered by $\DD^n\times z$.
By \refthm{MatsumotoDichotomy}, either the characteristic set $K_z$ of $L_z$ is a singleton (Type I) or the whole $S_\infty^{n-1}$ (Type II) for $\nu$-a.e. $z\in Z$. Here, $\nu$ is the pushforward  measure on $Z$, obtained by \refthm{localStructure}.

Denote by $\cC(S_\infty^{n-1})$ the space of closed subsets of $S_\infty^{n-1}$, equipped with the $\sigma$-algebra $\cB_\cC$ of the Chabauty topology (see, e.g., \cite[Section~E.1.]{Benedetti} for the basic theory of the Chabauty topology). 
From the above observation, we have an assignment $\frak{m}:Z \to \cC(S_\infty^{n-1}) $ given by 
$z\mapsto K_z$.
We call $\frak{m}$ the \emph{Matsumoto map} of $(\rho, m)$.
In \cite[Section~6]{Matsumoto12}, Matsumoto observed the following properties of $\frak{m}$.
\begin{lem}[\cite{Matsumoto12}]\label{Lem:keyProperty}
The following properties hold:
\begin{itemize}
    \item $\frak{m}$ is $\Gamma$-equivariant with respect to $\rho$ and the $\Gamma$-action on the boundary $S_\infty^{n-1}$, namely,
    \[
    \frak{m}(\rho(\gamma)\cdot z)=\gamma \cdot \frak{m}(z) ,\quad\text{for all }\gamma\in\Gamma,\quad \nu\text{-a.e. }z\in Z.
    \]
    \item $\frak{m}$ is $\nu$-measurable with respect to $\cB_\cC$.
\end{itemize}
\end{lem}

\begin{const}[Matsumoto map of Type I]\label{Const:MatsumotoMap}
When $m$ is of Type $I$, $\frak{m}$ provides a map $\frak{n}:Z\to S_\infty^{n-1}$ such that there is a $\nu$-conull set $A\subset Z$ on which $\frak{n}$ is $\Gamma$-equivariant, as follows:
for a $\nu$-a.e. $z\in Z$, $K_z$ is a singleton, that is, there is a conull set $A$ such that $K_z=\{s_z\}$ for some $s_z\in S_\infty^{n-1}$.
By \reflem{keyProperty}, we can assume that $A$ is $\Gamma$-invariant.
Now, we define a $\Gamma$-equivariant map $\frak{n}$ from $A$ to $S_\infty^{n-1}$ by $z\mapsto s_z$. 
Then, by assigning  $Z\setminus A$ to an arbitrary point in $S_\infty^{n-1}$, we can extend the domain of $\frak{n}$, and  obtain a desired map $\frak{n}:Z\to S_\infty^{n-1}$.
If there is no confusion, then, by abusing the notation, we also denote $\frak{n}$ by $\frak{m}$ and call it the \emph{Matsumoto map}. Also, we call such a conull set $A$ a \emph{domain of equivariance} of  $\frak{m}$. 
\end{const}

Under the identification $s\in S_\infty^{n-1}$ with $\{s\}\in\mathcal{C}(S_\infty^{n-1})$, we can embed $\iota:S_\infty^{n-1}\hookrightarrow\mathcal{C}(S_\infty^{n-1})$ as a measurable subspace with the $\sigma$-algebra on $S_\infty^{n-1}$ induced from $\cB_\cC$.
Then, by the second property in \reflem{keyProperty} and the following lemma, we can see that $\frak{m}:Z\to S_\infty^{n-1}$ is $\nu$-measurable.

\begin{lem}\label{Lem:measurableEmb} 
Let $(X,d)$ be a compact metric space and $\cC(X)$ the space of closed subsets of $X$, equipped with the Borel $\sigma$-algebra of the Chabauty topology.
Under the identification $x\in X$ with $\{x\}\in\mathcal{C}(X)$, we can embed $\iota:X\hookrightarrow\mathcal{C}(X)$ as a measurable subspace with the $\sigma$-algebras on $X$, induced from the $\sigma$-algebra on $\cC(X)$.
Then, any open set $U\subset X$ is a Borel subset $\iota(U)\subset\mathcal{C}(X)$. In particular, the Borel $\sigma$-algebra of $\mathcal{C}(X)$ contains the Borel $\sigma$-algebra of $X$. 

\end{lem}
\begin{proof}

Note that since $(X,d)$ is a compact metric space, it follows from  \cite[Proposition E.1.3]{Benedetti} that 
$\mathcal{C}(X)$ is equipped with the Chabauty topology which is a metric topology induced by the Hausdorff distance associated with $d$.
It follows from \cite[Proposition E.1.2]{Benedetti} that if we choose a convergent sequence $X_n$ in $\mathcal{C}(X)$ such that $X_n =\{x_n\}$ and $X_n\to A$ in $\mathcal{C}(X)$, then $A$ is a singleton. This implies $\iota:X\hookrightarrow\mathcal{C}(X)$ is a closed embedding.
Now, if $U\subset X$ is an open subset, we can write \[\iota(U) = \Big(\mathcal{C}(X)\setminus \big(\iota(X)\setminus\iota(U)\big)\Big)\setminus \big(\mathcal{C}(X)\setminus\iota(X)\big).\] 
Since $\iota$ is a closed embedding, both $\mathcal{C}(X)\setminus \big(\iota(X)\setminus\iota(U)\big)$ and $\mathcal{C}(X)\setminus\iota(X)$ are open sets in $\mathcal{C}(X)$. 
Therefore, $\iota(U)$ is a Borel subset of $\mathcal{C}(X)$.
\end{proof}

\begin{cor}
    Any Matsumoto map given by \refconst{MatsumotoMap} is $\nu$-measurable.
\end{cor}
\begin{proof}
    Let $\frak{n}:Z\to S_\infty^{n-1}$ a Matsumoto map given by \refconst{MatsumotoMap}.
    Say that $A$ is a domain of equivariance and $\frak{n}(Z\setminus A)=p$ for some $p\in S_\infty^{n-1}$.
    Choose any open subset $U$ in $S_\infty^{n-1}$. 
    By \reflem{measurableEmb}, $\iota(U)$ is in the Borel $\sigma$-algebra $\cB_\cC$ of $\cC(S_\infty^{n-1})$. 
    Since, by \reflem{keyProperty}, the original Matsumoto map $\frak{m}:Z\to \cC(S_\infty^{n-1})$ is $\nu$-measurable with respect to $\cB_\cC$, the preimage $\frak{m}^{-1}(\iota(U))$ is $\nu$-measurable.

    If $p\notin U$, then 
    we have 
    \[
    \frak{n}^{-1}(U)=\frak{m}^{-1}(\iota(U)) \setminus (Z\setminus A)
    \]
     and so $\frak{n}^{-1}(U)$ is $\nu$-measurable.
     Otherwise, $p\in U$ and
     \[
     \frak{n}^{-1}(U)=\frak{m}^{-1}(\iota(U)) \cup (Z\setminus A).
     \]
     In this case, $\frak{n}^{-1}(U)$ is also $\nu$-measurable.
     Thus, $\frak{n}$ is $\nu$-measurable.
\end{proof}

Now, we can summarize the above results as follows.
\begin{cor}\label{Cor:measurability}
Let $\Gamma$ be a discrete subgroup of $\Isomp(\DD^n)$ such that $\DD^n/\Gamma$ is a closed hyperbolic manifold, and $\rho$ an action of $\Gamma$ on a closed manifold $Z$, that is, $\rho:\Gamma \to \Homeo(Z)$.
Let $m$ be an ergodic harmonic measure on the compact hyperbolic $C^2$ foliation $(M_\rho,\cF_\rho,g)$ given by the suspension of $\rho$. Then, if $m$ is of Type I, then the Matsumoto map of $(\rho,m)$, $\frak{m}:Z\to S_\infty^{n-1}$ (in the sense of \refconst{MatsumotoMap}),  is $\nu$-measurable where $\nu$ is given by \refthm{localStructure}.
\end{cor}
Recall the generalized Lusin's theorem (\cite[Theorem 2B]{Fremlin81}; see also \cite[Theorem 7.14.25]{Bogachev}).

\begin{thm}[\cite{Fremlin81}]\label{Thm:continuity}
Let $\mu$ be a Radon measure on a topological space $X$ and let $Y$ be a metric space. A mapping $\varphi:X\to Y$ is measurable with respect to $\mu$ if and only if it is almost continuous.
\end{thm}
Combining \refcor{measurability} with the above theorem, we conclude that the Matsumoto map is almost continuous.
\begin{cor}\label{Cor:continuity}
The Matsumoto map $\frak{m}$ given by \refconst{MatsumotoMap} is continuous $\nu$-a.e.
\end{cor}

\section{Discontinuity of Matsumoto maps}

Let $B_\Gamma$ be a closed hyperbolic $n$-manifold such that $\Gamma\leq \Isomp(\DD^n)$ and $B_\Gamma=\DD^n/\Gamma$ (as in the previous section).
For $N\in \NN$, we say that a subgroup $G$ of $\Homeop(S^1)$ \emph{has at most $N$ fixed points} if each non-trivial element of $G$ has at most $N$ fixed point. 
In this section, we discuss about the Matsumoto maps associated wtih an action $\rho$ of $\Gamma$ on the circle $S^1$ under the following conditions:
\begin{itemize}
    \item $\rho(\Gamma)$ is non-discrete in $\Homeop(S^1)$, and
    \item $\rho(\Gamma)$ has at most $N$ fixed points.
\end{itemize} 
In the end, we prove \refthm{indiscreteType2}.

In the proof of \refthm{indiscreteType2},  we make use of the results of \cite{BCT24} and some basic properties of circle actions. 
Hence, we first summarize the basic facts about the group actions on the circle and the results of \cite{BCT24}.

Recall that $\Homeop^{(k)}(S^1)$ is a $k$-fold covering of $\Homeop(S^1)$ such that the following diagram commutes for all $g\in\Homeop^{(k)}(S^1)$:
\[\begin{tikzcd}
	{\RR/k\ZZ} & {\RR/k\ZZ} \\
	{S^1} & {S^1}
	\arrow["g", from=1-1, to=1-2]
	\arrow["{p_k}"', from=1-1, to=2-1]
	\arrow["{p_k}", from=1-2, to=2-2]
	\arrow["{q_k(g)}", from=2-1, to=2-2]
\end{tikzcd}\]
where $p_k:\RR/k\ZZ\to\RR/\ZZ = S^1$ is the natural $k$-fold cover.
$\PSL^{(k)}(\RR)$ is the fiber product given by the following commutative diagram where the lower row is a central extension:
\[\begin{tikzcd}
	&& {\PSL^{(k)}(\RR)} & {\PSL(\RR)} \\
	1 & {\ZZ/k\ZZ} & {\Homeop^{(k)}(S^1)} & {\Homeop(S^1)} & 1
	\arrow[from=1-3, to=1-4]
	\arrow[from=1-3, to=2-3]
	\arrow[from=1-4, to=2-4]
	\arrow[from=1-4, to=2-5]
	\arrow[from=2-1, to=2-2]
	\arrow[from=2-2, to=1-3]
	\arrow[from=2-2, to=2-3]
	\arrow["{q_k}", from=2-3, to=2-4]
	\arrow[from=2-4, to=2-5]
\end{tikzcd}\]

The following theorem is the combination of \cite[Lemma~4.3.]{BCT24} and \cite[Theorem~C.]{BCT24}.
\begin{thm}\label{Thm:BCT}
    Let $G<\Homeop(S^1)$ be a non-elementary, non-discrete subgroup with at most $N$ fixed points. Then, there exists $k\geq 1$ such that $G$ is conjugate to a dense subgroup of $\PSL^{(k)}(\RR)$, that is, $fGf^{-1}$ is a dense subgroup of $\PSL^{(k)}(\RR)$ for some $f\in \Homeop(S^1)$.
\end{thm}
Given a group action $\rho:G\to \Homeop(S^1)$, we denote by $\Lambda(\rho)$ the \emph{limit set} of $\rho$, which is defined as 
\[
\Lambda(\rho)=\bigcap_{s\in S^1} \closure{\rho(G)\cdot s}.
\]
Recall the following statement from \cite[Lemma 8.5.]{KKM19}
\begin{lem}\label{Lem:finiteCoverLimitSet}
    Let $G$ be a finitely generated group, and let $1\leq k<\infty$. If there exists a commutative diagram
    \[\begin{tikzcd}
	& {\mathrm{Homeo}_+^{(k)}(S^1)} \\
	G & {\mathrm{Homeo}_+(S^1)}
	\arrow["{q_k}", from=1-2, to=2-2]
	\arrow["\rho", from=2-1, to=1-2]
	\arrow["{\rho_\circ}"', from=2-1, to=2-2]
\end{tikzcd}\]
    then we have
    $$\Lambda(\rho) = p_k^{-1}\circ\Lambda(\rho_\circ).$$
\end{lem}

Given a bijection $f$ on a set $X$, we denote the set of all fixed points of $f$ by $\Fix(f)$.

Now, we are ready to prove \refthm{indiscreteType2}. We first show the following lemma, which,  together with \refthm{continuity},
implies the non-existence of Type I Matsumoto maps in the sense of \refconst{MatsumotoMap}. 

\begin{lem}\label{Lem:discontinuity}
Let $\Gamma$ be a discrete subgroup of  $\Isomp(\DD^n)$ such that $\DD^n/\Gamma$ is a closed hyperbolic manifold, and $\rho:\Gamma\to \Homeop(S^1)$ a faithful action with non-discrete image. 
Assume that $\rho(\Gamma)$ has at most $N$ fixed points.
If $\varphi$ is a map (which is not necessarily continuous) from a subset $A\subset S^1$ to the ideal boundary $S_\infty^{n-1}$ of $\DD^n$  satisfying the following properties:
\begin{itemize}
    \item $A$ is $\rho(\Gamma)$-invariant, that is, $\rho(g)(A)=A$ for all $g\in \rho(\Gamma)$;
    \item $\varphi(\rho(g)(z))=g(\varphi(z))$ for all $z\in A$ and $g\in \Gamma$, 
\end{itemize}
then $\varphi$ is discontinuous at every point in $A$.
\end{lem}
\begin{proof}
Say $G = \rho(\Gamma)<\Homeop(S^1)$. 
Note that $\rho(\Gamma)$ is non-elementary, that is, there is no $\rho(\Gamma)$-invariant measure.
By \refthm{BCT} and the assumption, $G$ is conjugate to a dense subgroup of $\PSL^{(k)}(\RR)$, i.e., there is $f\in\Homeop(S^1)$ such that $fGf^{-1}$ is a dense subgroup of $\PSL^{(k)}(\RR)<\Homeop(\RR/k\ZZ)$. Since $q_k:\PSL^{(k)}(\RR)\to\PSL(\RR)$ is a group homomorphism which is also a $k$-fold covering map, we conclude $q_k(fGf^{-1})$ is a dense subgroup of $\PSL(\RR)$.
    
From the given representation $\rho:\Gamma\to\Homeop(S^1)$, we have an induced faithful representation $\rho_\circ:\Gamma\to\PSL^{(k)}(\RR)<\Homeop(S^1)\equiv \Homeop(\RR/k\ZZ)$ by $\rho_\circ = f\rho f^{-1}$. Put $\rho' = q_k\circ\rho_\circ:\Gamma\to\PSL(\RR)$. Since $\rho'$ is a dense subgroup of $\PSL(\RR)$ and $G$ is torsion-free, there is an element $\gamma\in\Gamma$ such that $\rho'(\gamma)\in\PSL(\RR)$ is an elliptic element of infinite order. 
    
Let $H = \langle\gamma\rangle<\Gamma$ and consider the restriction $\rho'|_H = q_k\circ\rho_\circ|_H$ of the representation. Then by \reflem{finiteCoverLimitSet}, we have \[ \Lambda(\rho_\circ|_H) = p_k^{-1}\circ\Lambda(\rho'|_H)\]
Since $\Lambda(\rho'|_H) = S^1$, $\Lambda(\rho_\circ|_H) = S^1$. As before, we may choose another element $\eta\in\Gamma$ so that the loxodromic isometries $\gamma$ and $\eta$ do not share a fixed point on $S_\infty^{n-1}$, and $\Lambda(\rho_\circ|_K) = S^1$ for $K = \langle\eta\rangle$. 
    
Now, we fix the chosen $\gamma$ and $\eta$ in $\Gamma$ and write $\gamma_\circ = \rho_\circ(\gamma)$ and $\eta_\circ = \rho_\circ(\eta)$. Choose a point $p\in A$. 
Since $\Fix(\gamma)\neq\Fix(\eta)$ on $S_\infty^{n-1}$, either $\varphi(p)\notin\Fix(\gamma)$ or $\varphi(p)\notin\Fix(\eta)$. So we may assume $\varphi(p)\notin\Fix(g)$. Since $\Lambda(\rho_\circ|_H)$ is dense, there is an increasing sequence $\{n_k\}_{k\in\NN}$ of positive numbers such that $\{\gamma_\circ^{n_k}(p)\}_{k\in\NN}$ is a sequence of distinct points converging to $p$. If $\varphi$ is continuous at $p$, then $\varphi(\gamma_\circ^{n_k}(p))\to\varphi(p)$ as $k\to\infty$. But $\varphi(\gamma_\circ^{n_k}(p)) = \gamma^{n_k}(\varphi(p))$ by assumption, $\{\varphi(\gamma_\circ^{n_k}(p))\}_{k\in\NN}$ converges to the attracting fixed point of $\gamma$. This contradicts the fact that $\varphi(p)\notin\Fix(\gamma)$. Therefore, $\varphi$ is discontinuous at $p$.  
    Since the choice of $p$ is arbitrary, we can obtain the desired result.
\end{proof}

\begin{cor}\label{Cor:faithfulCase}
Let $\Gamma$ be a discrete subgroup of  $\Isomp(\DD^n)$ such that $\DD^n/\Gamma$ is a closed hyperbolic manifold, and $\rho:\Gamma\to \Homeop(S^1)$ a faithful action with non-discrete image. 
Assume that $\rho(\Gamma)$ has at most $N$ fixed points. Then, any ergodic harmonic measure on the suspension foliation $(M_\rho,\cF_\rho,g)$ is of Type II.
\end{cor}
\begin{proof}
    Let $m$  an ergodic harmonic measure  on $(M_\rho,\cF_\rho,g)$. 
    Assume that $m$ is of Type I. 
    By \refconst{MatsumotoMap}, we have the Matsumoto map $\frak{m}:S^1 \to S_\infty^{n-1}$ which is $\Gamma$-equivariant on some domain of equivariance and $\nu$-measurable (\refcor{measurability}).
    By \refcor{continuity}, $\frak{m}$ is continuous $\nu$-a.e., but this contradicts \reflem{discontinuity}.
    Thus, $m$ is of Type II.
\end{proof}

Then, we  prove \refthm{indiscreteType2} by proving the case of non-faithful actions.

\begin{restate}{Theorem}{Thm:indiscreteType2}
Let $\Gamma$ be a discrete subgroup of  $\Isomp(\DD^n)$ such that $\DD^n/\Gamma$ is a closed hyperbolic manifold, and $\rho:\Gamma\to \Homeop(S^1)$ an action of $\Gamma$ on $S^1$.
Assume that $\rho$ satisfies one of the following conditions:
\begin{itemize}
    \item $\rho$ is non-faithful or
    \item $\rho$ is a faithful action such that $\rho(\Gamma)$ is a non-discrete subgroup in $\Homeop(S^1)$, having at most $N$ fixed points
\end{itemize}
Then, any ergodic harmonic measure on the suspension $(M_\rho,\cF_\rho,g)$ of $\rho$ is of Type II.
\end{restate}
\begin{proof}
    By \refcor{faithfulCase}, it suffices to show the case where $\rho$ is non-faithful. 
    To prove this case, we just follow the argument of \cite[Proposition~6.6.]{Matsumoto12}.

    Assume that $\rho$ is non-faithful and $m$ is an ergodic harmonic measure on the compact hyperbolic $C^2$ foliation $(M_\rho,\cF_\rho,g)$ given by the suspension of $\rho$.   
    If $m$ is completely invariant, that is, its characteristic harmonic function is constant on an $m$-a.e. leaf of $\cF_\rho$, then it is easy to see from the construction of the associated measure $[\mu_L]$ of $m$ that $m$ is of Type II.
    
    Then, we consider the case where $m$ is not completely invariant. 
    Assume for contradiction that $m$ is of Type I. Then, by \refconst{MatsumotoMap}, we have the Matsumoto map $\frak{m}:S^1 \to S_\infty^{n-1}$,  which is $\Gamma$-equivariant on a domain $A$ of equivariance and $\nu$-measurable (\refcor{measurability}). 
    Say that $\frak{m}(S^1\setminus A)=p$ for some $p\in S_\infty^{n-1}$.
    Here, $\nu$ is given by \refthm{localStructure}.
    Since $\nu$ is quasi-invariance (\refthm{wellDefine}),  
    the pushforward measure $\frak{m}_*\nu$ is also quasi-invariant under the $\Gamma$-action on the ideal boundary $S_\infty^{n-1}$. 
    Note that the support of $\nu$, $\supp(\nu)$, is an infinite set. 
    
    Choose a non-trivial element $\gamma \in \ker(\rho)$. Then, we can take a Borel fundamental domain $F$ for the $\langle \gamma \rangle$-action on $S_\infty^{n-1}\setminus \Fix(\gamma)$ since $\gamma$ is loxodromic and has exactly two fixed points: one attracting and one repelling.
    Also, we may assume that $p\notin F$.
    Then, we have $\nu(\frak{m}^{-1}(F))>0$. 
    
    On the other hand, since $\gamma\in \ker(\rho)$,  we have  
    \[
    \frak{m}^{-1}( F)=\rho(\gamma)\cdot\frak{m}^{-1}(F). 
    \]
    If $p\notin \gamma \cdot F$, then 
    \[\rho(\gamma)\cdot\frak{m}^{-1}(F)=\frak{m}^{-1}(\gamma \cdot F).
    \]
    Otherwise,  $p\in \gamma \cdot F$ and 
    \[
     (S^1\setminus A)\cup [\rho(\gamma)\cdot\frak{m}^{-1}(F)]=\frak{m}^{-1}(\gamma \cdot F) 
    \]
    Therefore, 
    we have 
    \[
    \frak{m}^{-1}(F) \cap \frak{m}^{-1}(\gamma \cdot F) =\frak{m}^{-1}(F).
    \]
    Then, 
    \[\nu(\emptyset)=\nu(\frak{m}^{-1}(F\cap \gamma \cdot  F))=\nu(\frak{m}^{-1}(\gamma \cdot F)\cap \frak{m}^{-1}(F))=\nu(\frak{m}^{-1}(F))>0.
    \]
    This is a contradiction. Thus, $m$ is of Type II.   
\end{proof}

For an immediate application, we note \cite[Example~2.112.]{Calegari07}.
\begin{ex}\label{Exa:GaloisConj}
Suppose that $M$ is a closed oriented hyperbolic 3-manifold, that is, $M = \Bbb H^3/\Gamma$ for some $\Gamma<\PSL(\Bbb C)$, which is a discrete subgroup isomorphic to $\pi_1(M)$. Mostow rigidity implies that $\Gamma$ is conjugate into $\PSL(K)$ for some number field $K$. 

Suppose that $K$ admits a real place so that there is a Galois embedding $\sigma:K\to\Bbb R$. Then $\sigma$ induces a faithful representation $\rho_{\sigma}:\Gamma\to\PSL(\RR)<\Homeop(S^1)$ by conjugating the representation $\Gamma\to \PSL(K)$. 
Note that $\rho_{\sigma}(\Gamma)$ is a non-discrete subgroup of $\PSL(\RR)$ by the $K(G,1)$-space property of $\Gamma$.
Since $\PSL(\RR)$ have at most $2$ fixed points,
by \refthm{indiscreteType2}, any ergodic harmonic measure on the compact hyperbolic $C^2$ foliation $(M_{\rho_\sigma},\rho_\sigma,g )$ is of Type II.
\end{ex}

\section{Application: Matsumoto dichotomy on foliated $S^1$-bundles over closed hyperbolic surfaces}

In this section, we complete the classification of the ergodic harmonic measures on the foliated circle bundles over closed hyperbolic surfaces in terms of the Matsumoto dichotomy.
To this end, we prove \refcor{classificationOverSurface}.

\begin{restate}{Corollary}{Cor:classificationOverSurface}
    Let $\Gamma$ be a subgroup of $\PSL(\RR)$ such that $S=\HH^2/\Gamma$ is a closed hyperbolic surface and let $(M_\rho,\cF_\rho, g)$ the compact hyperbolic $C^2$-foliation given by the suspension of an action $\rho:\Gamma\to \PSL(\RR)$.
    Then, the following holds.
    \begin{itemize}
         \item If $\rho$ is discrete and faithful, then there is a unique ergodic harmonic measure on $(M_\rho,\cF_\rho, g)$ that is of Type~I ;
        \item If $\rho$ is either with non-discrete image (equivalently, dense image) or non-faithful, then any ergodic harmonic measure on $(M_\rho,\cF_\rho, g)$ is of  Type~II.        
    \end{itemize}
\end{restate}
\begin{proof}
    Since $\PSL(\RR)$ has at most $2$ fixed points, the second statement is an immediate consequence of \refthm{indiscreteType2}.
    When $\rho$ is discrete and faithful, $(M_\rho, \cF_\rho, g)$ is the Anosov foliation on the unit tangent bundle of $S$. 
    Recall that the result of Matsumoto \cite[Theorem~5.7. and Example~5.8]{Matsumoto12}(or \cite[Theorem 3.9]{DK}, \cite[Proposition 5, p. 305]{Garnett}) implies that there is a unique harmonic measure on $(M_\rho, \cF_\rho, g)$.
    By the uniqueness, it is automatically ergodic.
    In this case, the unique harmonic measure is of Type I by \cite[Example~6.5.]{Matsumoto12}. 
\end{proof}

Unlike in \refcor{classificationOverSurface}, \refthm{indiscreteType2} does not cover the case of discrete and faithful actions.
Hence,  it is natural to ask the following question.
\begin{ques}\label{Que:DFTypeI}
    Let $\rho$ be a discrete and faithful action of a closed hyperbolic manifold group on $S^1$. 
    Is any ergodic harmonic measure on the suspension foliation of $\rho$ Type I?
\end{ques}

In fact, such actions naturally arise in the context of geometric topology.
See the following example.
\begin{ex}\label{Exa:fibering}
Recall from \refsec{intro} the action $\rho_{univ}:\pi_1(M_\varphi)\to \Homeop(S^1)$ of the fundamental group of the hyperbolic mapping torus on the ideal boundary of the universal cover of the fiber surface.
In fact, $\rho_{univ}$ is a discrete faithful action with at most $N$ fixed points. 
To see this, first note that the action $\rho_{univ}$ is the extension of the standard action $\Gamma\to \PSL(\RR)$ where $\HH^2/\Gamma$ is the fiber surface.
It is well known that any lifting $\phi$ of $\varphi$ induces an element in $\Homeop(S^1)$ with a finite orbit; see, e.g., \cite{HandelThurston}.
More precisely, some power of $\phi$ has an even number of fixed points.
Half of the fixed points are attracting, and the other half are repelling and there is some lifting $\phi$, some power of which admits at least $4$ fixed points.
In particular, we can see that $\rho_{univ}(\pi_1(M_\varphi))$ has at most $N$ fixed points and each non-trivial element of $\rho_{univ}(\pi_1(M_\varphi))$ admits a finite orbit; see, e.g., \cite{HandelThurston},\cite{CassonBleiler}.

On the other hand, $G=\rho_{univ}(\pi_1(M_\varphi))$ is discrete in $\Homeop(S^1)$. If not, by \refthm{BCT}, $fGf^{-1}$  would be a dense subgroup of $\PSL^{(m)}(\RR)$ for some $m\in\NN$ and some $f\in \Homeop(S^1)$.
Observe that $\PSL^{(k)}(\RR)$ contains an element that fixes exactly two points: one repelling and one attracting if and only if $k=1$.
Since $G$ contains $\Gamma$, $m=1$ and $fGf^{-1}\leq \PSL(\RR)$. However, it is a contradiction since there is a non-trivial element in $fGf^{-1}$ having at least $4$ fixed points.
This implies the discreteness of $\rho_{univ}(\pi_1(M_\varphi))$.
\end{ex}

\section{Application: some remarks on Amenable foliations}
In this section, we briefly discuss the following question,  posed by Matsumoto \cite[Question~6.7.]{Matsumoto12}
\begin{ques}\label{Que:amenable}
    It is known \cite{Kai} that a compact hyperbolic foliation with a type I ergodic harmonic measure is an amenable measured foliation in the sense of \cite[Section~3.3.]{AR}. Is the converse true?
\end{ques}
We may ask the above question for the suspension of an action satisfying some proper regularity. 
For instance, let $\rho:\pi_1(S)\to \PSL(\RR)$ an action of the fundamental group of a closed hyperbolic surface $S$ on the circle. Say that $(M_\rho,\cF_\rho,g)$ is the compact hyperbolic $C^2$-foliation given by the suspension of $\rho$.
If $(M_\rho,\cF_\rho,g)$ is amenable in the sense of \cite[Section~3.3.]{AR}, then we can easily see that $\rho$ should be faithful, i.e. $\ker(\rho)$ is trivial (\cite[Corollary~5.3.33.]{AR}).
By using \cite[Theorem~4.8.]{Moore}, which is a generalization of the result of \cite{GC}, we can also conclude that $\rho$ is discrete. 
Therefore, by the first statement in \refcor{classificationOverSurface}, we can see that \refque{amenable} is true for suspensions of $\PSL(\RR)$-representations of closed hyperbolic surface groups.

Motivated by \refthm{indiscreteType2},  and \refque{DFTypeI}, we may consider \refque{amenable} for actions of closed hyperbolic $3$-manifold groups on the circle with at most $N$ fixed points. 

\begin{ques}
   Does \refque{amenable} hold for the suspension foliation given by a discrete faithful action of a closed hyperbolic $3$-manifold group on the circle with at most $N$ fixed points?
\end{ques}
Alternatively, we may ask the following.
\begin{ques}
    Is there a discrete faithful action of a closed hyperbolic $3$-manifold group on $S^1$, the suspension of which is amenable in the sense of \cite[Section~3.3.]{AR} and admits a type II ergodic harmonic measure?
\end{ques}
Such an example will provide a counter-example of \refque{DFTypeI} and \refque{amenable}. 

\section*{Acknowledgement}
We would like to thank Yoshifumi Matsuda, Yoshihiko Mitsumatsu, Hiraku Nozawa and Michele Triestino for helpful conversations and comments.
The first author was supported by the National Research Foundation of Korea(NRF) grant funded by the Korea government(MSIT) (RS-2022-NR072395).
The second author was partially supported by the National Research Foundation of Korea (NRF) grant funded by the Korea government (MSIT) (No. 2020R1C1C1A01006912) and by Mid-Career Researcher Program (RS-2023-00278510) through the National Research Foundation funded by the government of Korea.

\bibliographystyle{alpha}
\bibliography{biblio.bib}

\end{document}